\documentclass[a4paper,11pt,reqno]{amsart}

\usepackage{amsmath,graphics}
\usepackage{amssymb}
\usepackage{amsfonts}
\usepackage{mathtools}
\usepackage{latexsym}
\usepackage{eucal}
\usepackage[dvips]{graphicx}

\usepackage[usenames,dvipsnames]{color}
\usepackage[colorlinks=true,linkcolor=Red,citecolor=Green]{hyperref}

\theoremstyle{plain}
\newtheorem{thm}{Theorem}[section]
\newtheorem{propo}[thm]{Proposition}
\newtheorem{lem}[thm]{Lemma}
\newtheorem{cor}[thm]{Corollary}
\newtheorem{conj}[thm]{Conjecture}

\theoremstyle{definition}
\newtheorem{defi}[thm]{Definition}
\newtheorem{rem}[thm]{Remark}

\newcommand{\R}{\mathbb{R}}
\newcommand{\C}{\mathbb{C}}

\newcommand{\pp}{{\mathbb P}}

\title[]{On the spectral gap of negatively curved surface covers}

\author[W. Hide, J.~Moy~and~F.~Naud]{Will Hide, Julien Moy and Fr\'ed\'eric Naud}
\address{Fr\'ed\'eric Naud\\
	Institut Math\'ematique de Jussieu\\
	Universit\'e Pierre et Marie Curie, 4 place Jussieu, 75252 Paris Cedex 05\\
	France.
}

\email{frederic.naud@imj-prg.fr}

\address{Julien Moy \\
	Laboratoire de math\'ematiques d'Orsay \\
	Universit\'e Paris-Saclay, 91405 Orsay Cedex\\France.}
\email{julien.moy@universite-paris-saclay.fr}
\address{ Will Hide\\
	University of Oxford
}
\email{william.hide@maths.ox.ac.uk}
\subjclass{}

\keywords{}

\begin{document}
 \bibliographystyle{plain}
	\maketitle
	
	\begin{abstract} Given a negatively curved compact Riemannian surface $X$, we prove explicit estimates, valid with high probability as the degree goes to infinity,
		of the first non-trivial eigenvalue of the Laplacian on random Riemannian covers of $X$. The explicit gap depends on the bottom of the spectrum of the universal cover $\lambda_0$.
		We first prove a lower bound which generalizes in {\it variable curvature} a result of Magee--Naud--Puder \cite{MNP} known for hyperbolic surfaces.
		We then formulate a conjecture on the optimal spectral gap and show that if the sequence of representations associated to the covers converge strongly, the conjecture holds.
	\end{abstract}
	
	\bigskip \tableofcontents
	
	\section{Introduction}
	
	Let $X$ be a smooth, compact, connected surface $X$ without boundary. We assume that $X$ is endowed with a smooth Riemannian metric $g$ whose Gaussian curvature $K_g$ is strictly negative on $X$. Let $\widetilde{X}$ denote the universal cover of $X$ with its induced Riemannian metric $\widetilde{g}$. The Riemannian cover $(\widetilde{X},\widetilde{g})$ is then a Cartan--Hadamard manifold \emph{i.e.} a complete, simply connected manifold, with strictly negative curvature, and the fundamental group $\Gamma$ of $X$ acts by isometries on $\widetilde{X}$ so that we can view $X$ as a quotient
	$$X=\Gamma \backslash \widetilde{X}.$$ 
	Let $\phi_n:\Gamma\rightarrow \mathcal{S}_n$ be a group homomorphism, where $\mathcal{S}_n$ is the symmetric group of permutations of $\{1,\ldots,n\}$.
	The group $\Gamma$ then acts on $\widetilde{X}\times \{1,\ldots,n\}$ via 
	$$\gamma.(z,j)\stackrel{\rm{def}}{=}(\gamma(z),\phi_n[\gamma](j)).$$
	The quotient $X_{n}\stackrel{\rm{def}}{=}\Gamma \backslash_{\phi_n} \widetilde{X}\times\{1,\ldots,n\}$ is then an $n$-sheeted, possibly non-connected, Riemannian cover of $X$. 
	Endowing the finite set $\mathrm{Hom}(\Gamma, \mathcal{S}_n)$ with the uniform probability measure gives a notion of {\it random $n$-sheeted covers of $X$}, and the corresponding probability measure will be denoted by $\pp_n$.
	Let $\Delta_{X_n}$ denote the positive Laplace--Beltrami operator on $X_n$, whose $L^2$-spectrum is a discrete subset of $\R^+$, starting at $0$. We will denote by $\Delta_X$ the Laplacian on the base surface $X$. Let $\lambda_1(X_n)$ denote the first non-trivial eigenvalue of $\Delta_{X_n}$ on $X_n$. It is a natural question, inspired by the existing huge literature on expander graphs, to understand {\it quantitatively} how $\lambda_1(X_n)$ behaves as $n\rightarrow \infty$.
	
	A qualitative fact can be derived as follows. Fixing a symmetric set of generators $S\stackrel{\rm{def}}{=}\{g_1,\ldots,g_k, g_1^{-1},\ldots,g_k^{-1}\}$ of $\Gamma$, we define a graph $\mathcal{G}(\phi_n)$ whose set of vertices is $\{1,\ldots,n\}$. For all $j\in \{1,\ldots,n\}$ we attach an edge from $j$ to each $\phi_n[g](j)$, where $g\in S$. Let $\lambda_1(\mathcal{G}(\phi_n))$ denote
	the smallest non-trivial eigenvalue of the combinatorial Laplacian on $\mathcal{G}_n$. From the work of Brooks 
	\cite[Theorem 1]{Brooks1} and an inequality of Alon and Milman \cite{AM}, see also in the book of Lubotzky \cite[Theorem 4.3.2]{Lu1}, 
	we know that if we have a sequence of graphs $(\mathcal{G}(\phi_{n_j}))_j$ such that for all $j$,
	$\lambda_1(\mathcal{G}(\phi_{n_j}))\geq \varepsilon_1$ for some uniform $\varepsilon_1>0$, then there exists $\varepsilon_2>0$ such that for all $j$, 
	$\lambda_1(X_{n_j})\geq \varepsilon_2$. In other words, if the family of graphs $(\mathcal{G}(\phi_{n_j}))_j$ is an expander family, then there exists a uniform lower bound on the spectral gap of the corresponding covers $X_{n_j}$. It is not difficult to show, using the combinatorial results from \cite[Theorem 1.11]{MNP}, that random graphs $\mathcal{G}_{\phi_n}$ form an expander family with high probability \emph{i.e.} there exists $\varepsilon_1>0$ such that 
	$$\lim_{n\rightarrow +\infty} \pp_n(\lambda_1(\mathcal{G}(\phi_n))\geq \varepsilon_1)=1.$$
	Therefore there exists a non-effective $\varepsilon_2>0$ such that $\lambda_1(X_n)\geq \varepsilon_2$ with high probability as $n\rightarrow +\infty$. The main purpose of this paper is to 
	provide explicit $\varepsilon_2>0$, in the spirit of the result of the recent results obtained in {\it constant} negative curvature.
	
	\bigskip Let 
	$\lambda_0=\lambda_0(\widetilde{X})$ be the bottom of the $L^2$-spectrum of the Laplacian on the universal cover $\widetilde{X}$. Since $\Gamma$ is non-amenable, by a result of Brooks \cite[Theorem 1]{BrooksAmenable}, one has automatically $\lambda_0>0$. More quantitatively, if we set $-a^2:=\sup_{z\in \widetilde{X}}K_{\widetilde{g}}(z)$, by a theorem of McKean, see \cite[Chapter 3, Theorem 4]{Chavel}, we know that 
	$$\lambda_0\geq \frac{a^2}{4}.$$
	
	Our first result is as follows.
	\begin{thm}
		\label{main1} For all $\varepsilon>0$, with high probability as $n\rightarrow +\infty$, we have
		$$\mathrm{Sp}(\Delta_{X_n})\cap \Big[0, \frac{\lambda_0}{2}-\varepsilon   \Big]= \mathrm{Sp}(\Delta_X)\cap \Big[0, \frac{\lambda_0}{2}-\varepsilon   \Big],$$
		where the multiplicities coincide on both sides.
	\end{thm}
	As we will explain below, Thm \ref{main1} is sub-optimal, and one can do much better. The proof we give in this paper is nevertheless interesting and contains
	counting bounds that are new and of independent interest.
	
	In the hyperbolic case, we have $\lambda_0=1/4$, which gives a relative spectral gap of size $1/8-\varepsilon$. In \cite{MNP}, using Selberg's trace formula, Magee, Naud and Puder had obtained $3/16-\varepsilon$. By analogy with the existing literature for hyperbolic surfaces, it is natural to make the following conjecture.
	
	\begin{conj}
	\label{MC1}
		For all $\varepsilon>0$, with high probability as $n\rightarrow +\infty$, we have
		$$\mathrm{Sp}(\Delta_{X_n})\cap \left [0, \lambda_0-\varepsilon   \right]=
		\mathrm{Sp}(\Delta_X)\cap \left [0, \lambda_0-\varepsilon   \right],$$
		where the multiplicities coincide on both sides.
	\end{conj}
	We point out that this conjecture has been proved in the hyperbolic case by Hide and Magee for random {\it finite area} surface covers, where $\lambda_0=1/4$, see \cite{HM}. In the compact hyperbolic case, a $3/16-\varepsilon$ gap was first proved in \cite{MNP}. For other probabilistic models such as the Weil--Petersson moduli space in the large genus regime, $3/16-\varepsilon$-theorems have been obtained independently by Wu and Xue \cite{WX} and Lipnowski and Wright \cite{LW}. In a recent series of breakthrough preprints, Anantharaman and Monk \cite{AnMo1, AnMo2, AnMo3} have managed to improve the bound to $1/4-\varepsilon$.	
	The upper bound for the relative spectral gap given by the bottom of the spectrum $\lambda_0$ on the universal cover is optimal, even in variable curvature. Indeed, in \cite[Corollary 4.2]{Donnelly1} the following fact is proved. Assume that $\Gamma_j\subset \Gamma$ is a decreasing sequence of subgroups such that each $X_j:=\Gamma_j \backslash \widetilde{X}$ is compact and $\cap_j \Gamma_j=\{\mathrm{id}\}$. Then setting
	$$N_j(\lambda):=\#\{ \alpha \in \mathrm{Sp}(\Delta_{X_j})\ :\ \alpha \leq \lambda\},$$
	we have for all $\lambda>\lambda_0$,
	$$\liminf_{j\rightarrow +\infty}  \frac{N_j(\lambda)}{\mathrm{Vol}(X_j)}>0.$$
	By slightly modifying the arguments in \cite{Donnelly1}, see for example \cite{LS1} in constant negative curvature, it is not difficult to show that for all $\varepsilon>0$, with high probability as $n\rightarrow +\infty$, we have $\lambda_1(X_n)\leq \lambda_0+\varepsilon$.  
	
Let us denote by $V_{n}^{0}\stackrel{\rm{def}}{=} \ell_{0}^{2}\left(\left\{ 1,\dots,n\right\} \right)$ 
	the space of square summable functions with $0$-mean. For $\phi\in\textup{Hom}\left(\Gamma,\mathcal S_{n}\right)$
	we set \begin{equation}
		\rho_{\phi}\stackrel{\rm{def}}{=}\text{std}_{n-1}\circ\phi \label{eq:factor-th}
	\end{equation}
	on $V_{n}^{0}$
	where $\text{std}_{n-1}$ is the standard $n-1$ dimensional irreducible
	representation. 	To any sequence of covers $X_{n_k}$ of a base surface $X$, we denote by $\rho_k:=\rho_{\phi_{n_k}}$ the associated sequence of representations. 
	Let $\C[\Gamma]$ denote the group ring of $\Gamma$. The sequence $(\rho_k)$ is said to {\it converge strongly} if for any $z\in \C[\Gamma]$, we have in operator norm
	$$\lim_{k\rightarrow +\infty} \Vert \rho_k(z)\Vert_{V_n^0}=\Vert \lambda_\Gamma(z) \Vert_{\ell^2(\Gamma)},$$
	where $\lambda_\Gamma:\Gamma\rightarrow \mathcal{U}(\ell^2(\Gamma))$ is the right regular representation. In other words, in the large cover regime, the operator norm of the sum
	of matrices  $\rho_k(z)$ converges to the operator norm of the corresponding sum of convolution operators on $\ell^2(\Gamma)$. For a smooth introduction to strong convergence, we refer
	the reader to \cite{Magee_survey}.
	In this paper, we will show the following.
	
	\begin{thm}\label{main2} Fix a base surface $X$ as above and $\varepsilon>0$. Assume that the sequence of representations $(\rho_k)$ converges strongly, then for  all $k$ large, we have
				$$\mathrm{Sp}\big(\Delta_{X_{n_k}}\big)\cap \left [0, \lambda_0-\varepsilon   \right]=
		\mathrm{Sp}(\Delta_X)\cap \left [0, \lambda_0-\varepsilon   \right]. $$
	\end{thm}
	We know that by a result of Louder-Magee \cite{louder2022strongly}, we know that over any base surface $X$, 
	there exists at least a sequence of strongly  convergent representations and hence a sequence of covers with near optimal spectral gap. If one can prove
	\footnote{Shortly after we completed the writing, Michael Magee, Doron Puder and Ramon Van Handel posted a preprint \cite{MPV} proving this fact, so it seems that Conjecture \ref{MC1} is now a Theorem.} that random representations $\rho_{\phi_n}$ converge strongly with probability tending to one as $n\rightarrow +\infty$, Conjecture \ref{MC1} is proved.

	\bigskip
	The main obstacle in establishing Theorems \ref{main1}, \ref{main2} is mostly of analytic nature. 
	Indeed, the combinatorial tools remain the same as in constant negative curvature. However, the lack of an exact trace
	formula in variable negative curvature makes it a more challenging task. Semi-classical trace formulae are available, but are effective in a high frequency regime only: the problem at hand is typically a low frequency one. In   this paper, we rely on heat and Green kernel techniques which are robust enough to handle the variable curvature. 
	
	\bigskip
	The paper is organized as follows. In $\S $\ref{sec: heat kernels}, we review some facts on the heat kernel on the universal cover. Although there is no explicit formula, some Gaussian upper bounds on the heat kernel are available on Cartan--Hadamard manifolds. We then derive a formula (Proposition \ref{Traceform1}) which expresses the heat trace on random covers $X_n$, and which is the main tool in our analysis and serves as a substitute to Selberg's formula.  In $\S$\ref{sec: counting}, we recall necessary facts about the isometry group $\Gamma$. Section $\S$\ref{sec: mainproof} contains the main proof of Theorem \ref{main1}. It is divided into two steps where we have to distinguish between summation over primitive elements, where Gaussian bounds are good enough to produce a non-trivial bound, and non-primitive elements for which a much more subtle argument is required: we rely on powerful Green kernel estimates proved by Ancona and Gou\"ezel for the purpose of random walks, see below for details and references. In the last section, we prove Theorem \ref{main2} using techniques of strong convergence. A representation theoretic result of Louder--Magee \cite{louder2022strongly}
	shows that any abstract surface group admits sequences of symmetric representations which converge strongly (see in the section for details). This strongly convergent sequence allows estimating the operator norm of the heat semi-group on the associated covers (and hence the size of the spectral gap), again using a combination of heat kernels estimates and similar ideas as in \cite{HM}. It is rather striking that heat kernels techniques allow one to recover the near-optimal estimate proved in \cite{HM, louder2022strongly} which was using heavily an explicit parametrix
	of the resolvent of the Laplacian and explicit computations of Selberg transform.  
	
	\bigskip
	\noindent {\bf Acknowledgements.}
	We thank Dima Jakobson for some motivating discussions at the early stages of this project. We also thank Fr\'ed\'eric Paulin and Gilles Courtois for helpful answers to some of our questions.
	\section{Heat Kernels on random covers} \label{sec: heat kernels}
	We start by recalling some general facts on heat kernels, see for example in \cite[Chapter 8.2]{Chavel}. Let $M$ be a complete Riemannian manifold whose Ricci curvature
	is bounded from below. Let $\mathrm dy$ denote the Riemannian measure on $M$. There exists a unique positive function (the heat kernel)
	$$H_M\in C^\infty((0,\infty)\times M\times M),$$
	such that for any continuous, compactly supported function $\varphi \in C_c(M)$, the function
	$$u(t,x)\stackrel{\rm{def}}{=}\int_M H_M(t,x,y)\varphi(y)~\mathrm dy,$$
	solves the linear PDE
	$$\frac{\partial}{\partial t}u=-\Delta u,$$
	and satisfies pointwise
	$$\lim_{t\rightarrow 0^+} u(t,x)=\varphi(x).$$
	Here $\Delta$ is the positive Laplace-Beltrami operator on $M$.
	In general, there is no explicit formula for $H_M$, but under some reasonable curvature assumptions, so-called Gaussian upper bounds are known. In particular,
	(see for example \cite[Inequality 6.40]{Grigoryan1}) since $\widetilde{X}$ is a Cartan--Hadamard manifold, for any $D>4$, there exists a constant $C>0$ such that for all $t>0$, $\widetilde{x},\widetilde{y} \in \widetilde{X}$, 
	\begin{equation}
		\label{Gaussian}
		H_{\widetilde{X}}(t,\widetilde{x},\widetilde{y})\leq \frac{C}{\min\{1,t\}}\exp \left ( -\lambda_0(\widetilde{X})t-\frac{d(\widetilde{x},\widetilde{y})^2}{Dt}\right).
	\end{equation}
	
	\begin{rem} On the hyperbolic plane $\mathbb H$, one can show via explicit computations  (here $\lambda_0=\frac 14$ and $\delta=1$) that 
		\begin{equation}\label{bound heat kernel hyperbolic plane} H_{\mathbb H}(t,\tilde x,\tilde y)\le \frac{C}{t}\big(1+d(\tilde x,\tilde y)^2/t\big)\exp\Big(-\frac t4 -\frac{d(\tilde x,\tilde y)}{2}-\frac{d(\tilde x,\tilde y)^2}{4t}\Big).\end{equation}
		Notice, compared to \eqref{Gaussian}, the additional term $-\frac{d(\tilde x,\tilde y)}{2}$ in the exponential.\end{rem}
	
	\begin{rem}The pointwise estimate \eqref{bound heat kernel hyperbolic plane} allows recovering the fact that $H_{\mathbb H}(t,x,\cdot)$ has mass $\lesssim 1$. One cannot hope to achieve this in variable negative curvature using the upper bound \eqref{Gaussian}. Indeed, this upper bound depends only on the distance from $y$ to $x$, but the level sets of the heat kernel $H_{\widetilde X}(t,x,\cdot)$ in variable curvature may be very different from spheres centered at $x$.\end{rem}

	The heat kernel on each $n$-sheeted cover $X_n\to X$ can be expressed in terms of the heat kernel on $\widetilde X$ by averaging over the group $\Gamma$. We use the notation $[n]\stackrel{\rm{def}}{=}\{1,\ldots,n\}$. We recall that $\phi_n:\Gamma\rightarrow \mathcal{S}_n$ is a homomorphism that determines the cover $X_n$.
	\begin{lem}\label{lem: heat kernel cover} Denote by $\pi:\widetilde X\times [n]\to X_n$ the projection. Then, for any $x,y\in X_n$, if $(\tilde x,i)$ and $(\tilde y,j)$ are lifts of $x,y$ to $\widetilde X\times [n]$, one has
		\begin{equation}\label{eq: lem heat kernel cover}H_{X_n}(t,x,y)=\sum_{\gamma\in \Gamma, \ \phi_n[\gamma](j)=i} H_{\widetilde X}(t,\tilde x,\gamma \tilde y).\end{equation}\end{lem}
	
	\begin{proof} Recall that $X_n$ is defined as the quotient $\Gamma\backslash_{\phi_n} (\widetilde X\times [n])$, where $\gamma\in\Gamma$ acts by the formula $\gamma.(\tilde x,i)=(\gamma \tilde x,\phi_n[\gamma](i))$. The heat kernel on $X_n$ can therefore be expressed
		\begin{equation}\label{eq: heat kernel cover} H_{X_n}(t,x,y)=\sum_{\gamma\in \Gamma} H_{\widetilde X\times [n]}(t,(\tilde x,i),\gamma.(\tilde y,j)),\end{equation}
		where $H_{\widetilde X\times [n]}$ is the heat kernel on $\widetilde X\times [n]$ and  $(\tilde x,i)$, $(\tilde y,j)$ are any lifts of $x,y$ by the projection $\pi:\widetilde X\times [n]\to X_n$. The heat kernel on the manifold $\widetilde X\times [n]$ is simply given by
		\[H_{\widetilde X\times [n]}(t,(\tilde x,i),(\tilde y,j))=\delta_{ij} H_{\widetilde X}(t,\tilde x,\tilde y),\]
		where $\delta_{ij}=1$ if $i=j$ and $\delta_{ij}=0$ otherwise. Formula \eqref{eq: lem heat kernel cover} follows by injecting this expression in \eqref{eq: heat kernel cover}. \end{proof}
	
	For $t>0$, the operator $\mathrm{e}^{-t\Delta_X}$ has smooth kernel $H(t,x,y)$, therefore it is trace-class, and its trace its given by
	\[\operatorname{Tr} \mathrm{e}^{-t\Delta_X}=\sum_{\lambda_j\in \mathrm{Sp}(\Delta_X)} \mathrm{e}^{-t\lambda_j}=\int_X H_X(t,x,x)\mathrm{d}x.\]
	Let us now express the heat trace $\operatorname{Tr}(\mathrm{e}^{-t\Delta_{X_n}})$ on a cover $X_n\to X$. In the following, $\mathcal F$ denotes a bounded fundamental domain for the action of $\Gamma$ on $\widetilde X$.
	\begin{propo}
		\label{Traceform1}
		Let $F_n(\gamma)$ denote the number of fixed points of the permutation $\phi_n[\gamma]$. Then, the heat trace on $X_n$ is given by
		\[\operatorname{Tr}(\mathrm{e}^{-t\Delta_{X_n}})=\sum_{\gamma\in \Gamma} F_n(\gamma)\int_{\mathcal F}H_{\widetilde X}(t,\tilde x,\gamma \tilde x)\mathrm d\tilde x.\]\end{propo}
	
	\begin{proof}A fundamental domain for the action of $\Gamma$ on $\widetilde X\times [n]$ is given by
		\[\bigsqcup_{i=1}^n \mathcal F\times \{i\}.\]
		Therefore, if $f$ is a smooth function on $X_n$, one has
		\[\int_{X_n} f(x)\mathrm dx=\sum_{i=1}^n \int_{\mathcal F} f(\pi(\tilde x,i))\mathrm d\tilde x.\]
		Applying this formula to $f(x)=H_{X_n}(t,x,x)$, one obtains the following formula for the heat trace on $X_n$:
		\[\operatorname{Tr}(\mathrm{e}^{-t\Delta_{X_n}})=\sum_{\gamma\in \Gamma} \underbrace{\left( \sum_{i=1}^n \delta_{i,\phi_n[\gamma](i)}\right)}_{=F_n(\gamma)}\int_{\mathcal F}  H_{\widetilde X}(t,\tilde x,\gamma \tilde x) \mathrm d\tilde x,\]
		which concludes the proof.
	\end{proof}
	
	\section{Geometric preliminaries} \label{sec: counting}
	
	\subsection{Reminders on the isometry group $\Gamma$} We recall necessary facts about the fundamental group $\Gamma$ of $X$ and its link with closed geodesics of $X$, see \cite[Chapter 12]{doCarmo1992} for a short introduction. Recall that $\Gamma$ acts by deck transformations on $\widetilde X$, and that $X$ identifies with the quotient of $\Gamma\backslash\widetilde X$. We will denote by $d(z,w)$ the distance with respect to the metric $\widetilde{g}$ on the universal cover $\widetilde{X}$. 
	
	For $\gamma\in \Gamma$ different from the identity, consider the \emph{displacement function} $d_\gamma: \tilde x\mapsto d(\tilde x,\gamma \tilde x)$. The \emph{translation length} of $\gamma$ is 
	\[\ell(\gamma)\stackrel{\rm{def}}{=}\inf_{\tilde x \in \widetilde X}d(\tilde x,\gamma \tilde x)>0.\]
	The fact that $\ell(\gamma)$ is positive follows from the compactness of $X$. The displacement function attains its infimum on a single geodesic line called the \emph{axis of $\gamma$}, and accordingly denoted $\mathrm{axis}(\gamma)$. One can parameterize the axis of $\gamma$ by a unit-speed geodesic $t\in \mathbb R\mapsto \tilde \gamma(t)$ such that
	\[\gamma. \tilde \gamma(t)=\tilde \gamma(t+\ell(\gamma)).\]
	The path $\tilde \gamma_{|[0,\ell(\gamma))}$ therefore projects on an \emph{oriented} closed geodesic of length $\ell(\gamma)$ on $X$, that we will abusively denote by $\gamma$. The inverse $\gamma^{-1}$ has the same axis as $\gamma$, and acts by translation by $\ell(\gamma)$ on the line $t\mapsto \tilde \gamma(-t)$, which projects on the same closed geodesic as $\gamma$, but with reverse orientation.
	
	We denote by $\ell_0$ the length of the shortest closed geodesic of $X$. 
	\subsubsection{Closed geodesics and conjugacy classes of $\Gamma$}
	
	If $\gamma_1$ and $\gamma_2$ are two elements of $\Gamma$ such that $\gamma_2=g^{-1}\gamma_1g$ for some other element $g\in \Gamma$, then $\mathrm{axis}(\gamma_2)=g^{-1}\mathrm{axis}(\gamma_1)$ and $\ell(\gamma_2)=\ell(\gamma_1)$, so that $\gamma_1$ and $\gamma_2$ define the same oriented closed geodesic of $X$. Reciprocally, if $\gamma_1$ and $\gamma_2$ define the same closed geodesic, then they are conjugate in $\Gamma$.  Any closed geodesic can be obtained in such a way, so that the set of closed geodesics of $X$ identifies with the set of nontrivial conjugacy classes of $\Gamma$.

	\subsubsection{Primitive elements of $\Gamma$}
	A nontrivial isometry $\gamma\in \Gamma$ is \emph{primitive} if it cannot be written $\gamma=\gamma'^k$ for another $\gamma'\in \Gamma$ and $k\ge 2$. It is equivalent to the fact that the map $\tilde \gamma_{|[0,\ell(\gamma))}$ is injective, in which case the corresponding closed geodesic of $X$ is also called primitive. To any $\gamma\in \Gamma$ is associated a unique primitive element $\gamma_*$ such that $\gamma=\gamma_*^k$ for some $k\ge 1$. In this case, $\gamma$ and $\gamma_*$ share the same axis, and one has $\ell(\gamma)=k\ell(\gamma_*)$.
	
	The \emph{stabilizer} of $\gamma$, denoted $\mathrm{Stab}(\gamma)$, is the subgroup of isometries of $\Gamma$ that commute with $\gamma$. By \cite[Chapter 12, Lemma 3.5]{doCarmo1992}, $\mathrm{Stab}(\gamma)$ is a cyclic subgroup generated by $\gamma_*$, that is
	\[\mathrm{Stab}(\gamma)=\langle \gamma_*\rangle \stackrel{\rm{def}}{=}\{ \gamma_*^k~:~k\in \mathbf Z\}.\]
	
	From now on, we fix a subset $\mathcal P\subset \Gamma$ of primitive elements, such that any nontrivial element of $\Gamma$ is conjugated to a single power of an element of $\mathcal P$, that is, any $\gamma'\in \Gamma$ can be uniquely written
	\[\gamma'=g^{-1}\gamma^k g,\]
	for some $\gamma\in \mathcal P$, $k\ge 1$, and $g\in \langle \gamma\rangle\backslash \Gamma$. 
	
	\subsection{Counting results}
	The critical exponent of $\Gamma$, denoted by
	$\delta$, is by definition
	\begin{equation} \label{critical exponent}\delta\stackrel{\rm{def}}{=}\inf \left \{ s \in \R\ :\ \sum_{\gamma \in \Gamma} \mathrm{e}^{-sd(\gamma z,w)}<\infty \right \}.\end{equation}
	This critical exponent does not depend on $z,w$ and $\delta$ coincides with the {\it volume entropy} of the universal cover \emph{i.e.} we do have for all fixed $x\in \widetilde{X}$, 
	$$\lim_{r\rightarrow +\infty} \frac{\log(\mathrm{Vol}_{\widetilde{g}}(B(x,r)))}{r}=\delta, $$
	where $B(x,r)$ is the ball centered at $x$ of radius $r$ for the metric $\widetilde{g}$. Let us mention that $\delta$ also coincides with the {\it topological entropy} of the geodesic flow on the unit tangent bundle of $X$, by a result of Manning \cite{Man1}.
	
	The critical exponent $\delta$ defined in \eqref{critical exponent} is involved in various counting results obtained first by Margulis \cite{Margulis1969} in variable negative curvature. \begin{itemize}
		\item Let $x,y\in \widetilde X$, and let $\Gamma  y=\{\gamma y, \ \gamma\in \Gamma\}$ denote the orbit of $y$ under the action of $\Gamma$. There is a continuous function $c(x,y)$ such that
		\begin{equation}\label{eq: margulis} \# \big( \Gamma y\cap B(x,R)\big)\sim c(x,y) \mathrm{e}^{\delta R}, \qquad \text{as $R\to +\infty$}.\end{equation}
		
		\item Integrating $\eqref{eq: margulis}$ over a fundamental domain leads to the following estimate on the volume of balls of large radius. There is a continuous function $m(x)$ such that
		\begin{equation}\label{eq:growthvolume}\operatorname{Vol}_{\tilde g} B(x,R)\sim m(x) \mathrm{e}^{\delta r}, \qquad \text{as $R\to +\infty$}.\end{equation}
	\end{itemize}
	
	\subsection{Comparison results} If $X$ has sectional curvature $\le -\kappa^2$ everywhere, then its universal cover $\widetilde X$ is a $\mathrm{CAT}(-\kappa^2)$ space, according to the Cartan--Hadamard Theorem (see e.g. \cite[Chapter II, Theorem 4.1]{MNB}). Essentially, it means that any geodesic triangle $xyz$ in $\widetilde X$ is \emph{thinner} than any corresponding geodesic triangle $\overline{xyz}$ in the model space $\mathbb H_{(-\kappa^2)}$ of constant negative curvature $-\kappa^2$.
	
	We give two comparison results for geodesic triangles and quadrilaterals in $\widetilde X$.
	\begin{lem}\label{lem: comparison triangles} There is a constant $C=C(\kappa)$ such that for any geodesic triangle $xyz$ in $\widetilde X$ with $\angle_y =\frac{\pi}{2}$, one has
		\[d(x,z)\ge d(x,y)+d(y,z)-C.\]
	\end{lem}
	\begin{proof}Let $xyz$ be such a triangle, and consider a corresponding triangle $\overline{xyz}$ with same side lengths in the model space $\mathbb H_{(-\kappa^2)}$ of constant curvature $-\kappa^2$. Then, since $xyz$ is thinner than $\overline{xyz}$, the angle at $\overline y$ has to be $\ge \frac \pi 2$ (see \cite[Chapter II, Proposition 1.7]{MNB}). By moving $\overline x$ while keeping the lengths of $\overline{xy}$ and $\overline{yz}$ constant, we can ajust the angle $\angle_{\overline y}$ back to $\frac \pi 2$, but doing so decreases the length of $\overline{xz}$. We still denote this new triangle $\overline{xyz}$. Then, by a standard identity in hyperbolic trigonometry, one has
		\[\cosh(\kappa d(\overline x,\overline z))=\cosh(\kappa d(\overline x,\overline y))\cosh(\kappa d(\overline y,\overline z)).\]
		Since $d(x,z)\ge d(\overline x,\overline z)$ one finds
		\[\cosh( \kappa d(x,z))\ge \cosh(\kappa d(x,y))\cosh(\kappa d(x,z)).\]
		It follows that $\mathrm{e}^{\kappa d(x,z)/2}\ge \frac 14 \mathrm{e}^{\kappa(d(x,y)+d(x,z))/2}$ and one gets the result by taking logarithms on both sides.
	\end{proof}
	With similar ideas, one can prove the following comparison result for quadrilaterals.
	\begin{lem}\label{lem:paulinparkonnen} There is another constant $C=C(\ell_0,\kappa)$ such that for any $\gamma\in \Gamma\backslash\{\mathrm{id}\}$ and $x\in \widetilde X$, letting $r:=\mathrm{dist}(x,\mathrm{axis}(\gamma))$, one has
		\[d(x,\gamma x)\ge 2r+\ell(\gamma)-C\] \end{lem}
	\begin{proof}By  \cite[Lemma 9]{Parkkonen2015}, one has
		\[\sinh(\frac{\kappa}{2} d(x,\gamma x))\ge \cosh(\kappa r)\sinh(\frac \kappa 2 \ell(\gamma)).\]
		Since $\ell(\gamma)\ge \ell_0$ it follows that
		\[ \frac 12\exp\big(\frac{\kappa}{4}d(x,\gamma x)\big)\ge C(\kappa,\ell_0)\exp\big(\kappa (\frac r2 + \frac{\ell(\gamma)}{4})\big).\]
		It remains to take logarithms on both sides to conclude.\end{proof}
	
	\section{Proof of Theorem \ref{main1}} \label{sec: mainproof} The strategy to prove Theorem \ref{main1} is similar to that of \cite{MNP}, although we rely on heat kernel techniques instead of the Selberg trace formula.

	\begin{defi} An element of $\mathrm{Sp}(\Delta_{X_n})$ is called a \emph{new eigenvalue} if it appears with a higher multiplicity in the spectrum of $\Delta_{X_n}$ than in that of $\Delta_X$. The multiplicity of a new eigenvalue is the difference of its multiplicities in the spectra of $\Delta_{X_n}$ and $\Delta_X$.\end{defi}
	We let $\lambda_1(X_n)$ denote the first new eigenvalue of the Laplacian on $X_n$. Then, Theorem \ref{main1} is equivalent to the claim that
	\[\mathbb P_n\Big(\lambda_1(X_n)\le \frac{\lambda_0}{2}-\varepsilon\Big)\underset{n\to+\infty}\longrightarrow 0.\]
	\subsection{Strategy of the proof} Start from the observation
	\[\sum_{ \mathrm{new~eigenvalues}~\lambda } \mathrm{e}^{-t \lambda}=\operatorname{Tr}(\mathrm{e}^{-t\Delta_{X_n}}) -\operatorname{Tr}(\mathrm{e}^{-t\Delta_X}),\]
	where new eigenvalues are counted with multiplicities. Letting $\lambda_1(X_n)$ denote the first new eigenvalue of the Laplacian on $X_n$, we infer
	\[\mathrm{e}^{-\lambda_1(X_n)t}\le \operatorname{Tr}\big(\mathrm{e}^{-t\Delta_{X_n}}\big) -\operatorname{Tr}\big(\mathrm{e}^{-t\Delta_X}\big).\]
	Therefore, Theorem \ref{main1} will follow if one can prove that the difference between the traces of $\mathrm{e}^{-t\Delta_{X_n}}$ and $\mathrm{e}^{-t\Delta_X}$ is small on average. 
	
	Fix $\alpha\in (0,\lambda_0/2)$, then for any $t>0$ we have 
	$$\{\lambda_1(X_n)\le \alpha\}=\{\mathrm{e}^{-\lambda_1(X_n) t}\ge \mathrm{e}^{-\alpha t}\},$$ and by Markov inequality we get
	\[\mathbb P_n(\lambda_1(X_n)\le \alpha)\le \mathrm{e}^{\alpha t} \mathbb E_n(\mathrm{e}^{-\lambda_1(X_n) t})\le \mathrm{e}^{\alpha t}\mathbb E_n \Big[\operatorname{Tr}\big(\mathrm{e}^{-t\Delta_{X_n}}\big) -\operatorname{Tr}\big(\mathrm{e}^{-t\Delta_X}\big) \Big].\]
	Injecting the expressions of Proposition \ref{Traceform1} for the heat traces in the right-hand side, one gets
	\[\mathbb P_n(\lambda_1(X_n)\le \alpha)\le \mathrm{e}^{\alpha t}\sum_{\gamma \in \Gamma} \mathbb E_n\big[F_n(\gamma)-1\big]\int_{\mathcal F}H_{\widetilde X}(t,\tilde x,\gamma \tilde x)\mathrm d\tilde x.\]
	We split the right-hand side in three sums:
	
	\begin{itemize}
		\item The first only contains $\gamma=\mathrm{id}$. Since $F_n(\mathrm{id})=n$, using the pointwise upper bound \eqref{Gaussian} for the heat kernel on $\widetilde X$, we deduce that for $t\ge 1$,
		\begin{equation}\label{term1}\mathbb E_n\big[F_n(\mathrm{id})-1\big]\int_{\mathcal F}H_{\widetilde X}(t,\tilde x, \tilde x)\mathrm d\tilde x\lesssim n \mathrm{e}^{-\lambda_0 t},\end{equation}
		where the implied constant is independent of $n$ or $t$.
		\item The second sum is indexed over primitive elements of $\Gamma$. For such elements, one has good control on the expected value $\mathbb E_n\big[F_n(\gamma)-1\big]$, provided that the length of $\gamma$ is smaller than $C\log n$, for some arbitrary large (but fixed) constant $C$. We show in Proposition \ref{propo: contribution primitive elements} that if $t=\beta \log n$, then for any $\varepsilon'>0$,
		\begin{equation}\label{term2}\sum_{\gamma \ \rm{primitive}} \mathbb E_n\big[F_n(\gamma)-1\big]\int_{\mathcal F}H_{\widetilde X}(t,\tilde x,\gamma \tilde x)\mathrm d\tilde x\lesssim_{\beta,\varepsilon'} n^{\varepsilon'-1}.\end{equation}
		\item The third sum is indexed over non-primitive elements of $\Gamma$. For such elements, one has no control on the expected value $\mathbb E_n\big[F_n(\gamma)-1\big]$, which we shall trivially bound by $n$. We will use the following estimate on the nonprimitive contribution to the heat trace (see Proposition \ref{propo: contribution non primitive}):
		\begin{equation}\label{term3}\sum_{\gamma \ \rm{non~primitive}} \int_{\mathcal F}H_{\widetilde X}(t,\tilde x,\gamma \tilde x)\mathrm d\tilde x\lesssim t^3 \exp(-\lambda_0 t).\end{equation}
	\end{itemize} 
	Let us prove Theorem \ref{main1}, assuming the upper bounds \eqref{term1}--\eqref{term3}.
	
	\begin{proof}[Proof of Theorem \ref{main1}] Let $\alpha=\frac{\lambda_0}{2}-\varepsilon$, for some arbitrary small $\varepsilon>0$, and let $t=\frac{2}{\lambda_0} \log n$. Also, for convenience, let $\varepsilon'=\lambda_0^{-1}\varepsilon$, so that $\alpha t=(1-2\varepsilon')\log n$. Combining the estimates \eqref{term1}, \eqref{term2}, \eqref{term3}, one finds
		\[\mathbb P_n(\lambda_1(X_n)\le \alpha)\le C_{\varepsilon'}n^{-2\varepsilon'}.\]
		In particular, $ \mathbb P_n(\lambda_1(X_n)\le \alpha)$ goes to $0$ as $n$ goes to $+\infty$.\end{proof}
	
	It remains to prove the upper bounds \eqref{term2} and \eqref{term3}, which are the content of the following subsections.

	\subsection{The main combinatorial estimate}
	We recall that $\Gamma$, as an abstract group, is a {\it surface group}. Let $g\geq 2$ denote the genus of $X$, and denote by 
	$$\mathcal{A}=\{a_1,b_1,\ldots,a_g,b_g\}$$
	a set of generators of $\Gamma$. The group $\Gamma$ is then isomorphic to the abstract group
	$$\langle a_1,b_1,\ldots,a_g,b_g\ \vert \ [a_1,b_1][a_2,b_2]\ldots [a_g,b_g]\rangle.$$
	Given $\gamma \in \Gamma$, we denote by $\vert \gamma \vert_\mathcal{A}$ the word length of $\gamma$ with respect to the set of generators $\mathcal{A}$.
	Fix a base point $x_0\in \widetilde{X}$. Since $\Gamma$ act by isometries on the simply connected space $\widetilde{X}$ and is a co-compact group, a famous result of Svarc--Milnor, see \cite[Prop 8.19, page 140]{MNB}, tells us that there exist constants $\kappa_1,\kappa_2>0$ such that for all $\gamma \in \Gamma$,
	$$\vert \gamma \vert_{\mathcal A}\leq \kappa_1 d(\gamma x_0,x_0)+\kappa_2.$$
	Given a word $\gamma \in \Gamma$, we define by $\ell_w(\gamma)$ the cyclic word length of $\gamma$ i.e. the minimum of all word lengths $\vert g\vert_\mathcal{A}$ where $g$
	is in the conjugacy class of $\gamma$. The main probabilistic input from \cite{MNP} is the following fact.
	\begin{thm}
		\label{eq: asymptotique esperance fn}
		There exists a constant $A>0$, depending on $\Gamma$, such that for all $C>0$, for all primitive element $\gamma \in \Gamma$ with $\ell_w(\gamma)\leq C\log n$, we have as $n\rightarrow +\infty$,
		$$\mathbb{E}_n(F_n(\gamma))=1+O\left( \frac{\log(n)^A}{n}\right), $$
		where the implied constant in the $O(\bullet)$  depends only on $C$ and $\Gamma$.
	\end{thm}
	Notice that since we have
	$$\ell_w(\gamma)\leq \vert \gamma \vert_{\mathcal{A}}\leq  \kappa_1 d(\gamma x_0,x_0)+\kappa_2,$$
	the above asymptotic can be applied for all primitive words $\gamma \in \Gamma$ satisfying a bound of the type $d(\gamma x_0,x_0)\leq C\log(n)$ .

	\subsection{The contribution of primitive elements} \label{subsec: contribution primitive}
	
	Let $\Gamma_p\subset \Gamma$ denote the subset of primitive elements. We show
	
	\begin{propo}\label{propo: contribution primitive elements} Let $t=\beta \log n$ for some fixed $\beta$. Then, for any $\varepsilon> 0$,
		\[\sum_{\gamma \in \Gamma_{\rm p}} |\mathbb E_n(F_n(\gamma)-1)| \int_{\mathcal F}H_{\widetilde X}(t,x,\gamma x) \mathrm dx\lesssim_{\beta,\varepsilon} n^{\varepsilon-1}.\]\end{propo}
	
	\begin{proof}Let $C>0$, to be fixed later on. We will use Theorem \ref{eq: asymptotique esperance fn} to control the sums over finitely many words with $d(x_0,\gamma x_0)\leq C \log(n)$.
		For primitive elements with  $d(x_0,\gamma x_0)\geq C \log(n)$ we will only use the trivial bound $|\mathbb E_n(F_n(\gamma))-1|\le n$. This will not cause harm because $C$ can be chosen arbitrarily large (but fixed). We fix a point $x_0\in \mathcal F$ and let
		
		\[(\mathrm{I})\stackrel{\rm{def}}{=}\sum_{\gamma \in \Gamma_{\rm p}, d(x_0,\gamma x_0)\le C\log n} |\mathbb E_n(F_n(\gamma)-1))| \int_{\mathcal F}H_{\widetilde X}(t,x,\gamma x) \mathrm dx,\]
		\[(\mathrm{II})\stackrel{\rm{def}}{=}\sum_{\gamma \in \Gamma_{\rm p}, d(x_0,\gamma x_0)\ge C\log n} |\mathbb E_n(F_n(\gamma)-1))| \int_{\mathcal F}H_{\widetilde X}(t,x,\gamma x) \mathrm dx\]
		\subsubsection{Estimating $(\mathrm{I})$} 
		Elements $\gamma\in \Gamma$ involved in the sum $(\mathrm{I})$ satisfy 
		$$d(x_0,\gamma x_0)\le C\log n;$$ we can therefore invoke Theorem \ref{eq: asymptotique esperance fn} to get
		\[(\mathrm{I})\lesssim_{C} \frac{(\log n)^A}{n}\sum_{\gamma \in \Gamma_{\rm p}, d(x_0,\gamma x_0)\le C\log n} \int_{\mathcal F} H_{\widetilde X}(t,x,\gamma x)\mathrm dx.\]
		Now, since the heat kernel is positive, the sum in the right-hand side is bounded by
		\[\sum_{\gamma \in \Gamma} \int_{\mathcal F} H_{\widetilde X}(t,x,\gamma x)\mathrm dx,\]
		which is just the heat trace $\operatorname{Tr}(\mathrm{e}^{-t\Delta_X})$, that is uniformly bounded for $t\ge 1$. Consequently,
		\[(\mathrm{I}) \lesssim_C \frac{(\log n)^A}{n}\lesssim_{C,\varepsilon} n^{\varepsilon-1}.\]
		\subsubsection{Estimating $(\mathrm{II})$} \label{subsubsec: tail heat trace}
		First, note that if $x\in \mathcal F$ and $C$ is large enough then, whenever $d(x_0,\gamma x_0)\ge C \log n$, one has
		\[d(x,\gamma x)\ge d(x_0,\gamma x_0)-2 \mathrm{diam}(\mathcal F)\ge \frac{d(x_0,\gamma x_0)}{2}.\] Using a rough estimate $H_{\widetilde X}(t,x,y)\lesssim \mathrm{e}^{-d(x, y)^2/8t}$ for the heat kernel together with the trivial bound $\big|\mathbb E_n(F_n(\gamma)-1))\big|\le n$ then yields
		\[\mathrm{(II)}\lesssim n \sum_{\gamma \in \Gamma, d(x_0,\gamma x_0)\ge C\log n} \mathrm{e}^{-d(x_0,\gamma x_0)^2/16t}.\]
		We recall from \eqref{eq: margulis} that $\#\{\gamma \in \Gamma\,:\, d(x_0,\gamma x_0)\le R\}\lesssim \mathrm{e}^{\delta R}$. Since $t=\beta\log n$, the sum on the right-hand side is smaller than a constant times
		\[\int_{C \log n}^{+\infty} \exp\big(\delta r-\frac{r^2}{16t}\big)\mathrm dr \le \frac{1}{\frac{C}{16\beta}-\delta } \exp\Big((\delta-\frac{C}{16\beta})C\log n\Big).\]
		One can take $C=C(\delta,\beta)$ large enough to make the right-hand side smaller than $\frac 1n$.

	\end{proof}

	\subsection{The contribution of non-primitive elements} In this section, we prove the estimate on the nonprimitive contributions to the heat trace.
	
	\begin{propo}\label{propo: contribution non primitive} For any $t\ge 1$ we have
		\begin{equation}\label{eq: somme non primitif}\sum_{\gamma \in \Gamma_{\rm n.p.}} \int_{\mathcal F} H_{\widetilde X}(t,x,\gamma x)\mathrm dx \lesssim t^3 \exp (-\lambda_0 t).\end{equation}\end{propo}
		We first make a sequence of reductions. Fix $x_0\in \mathcal F$. Proposition \ref{propo: contribution non primitive} will follow from 
	\begin{propo}\label{propo: contribution non primitive truncated} For any $t\ge 1$ and $R\ge 1$, we have
		\begin{equation}\label{eq: somme non primitif truncated}\sum_{\gamma ~\mathrm{n.p.}, \ d(x_0,\gamma x_0)\le R} \int_{\mathcal F} H_{\widetilde X}(t,x,\gamma x)\mathrm dx \lesssim R^3 \exp (-\lambda_0 t).\end{equation}\end{propo}
	\begin{proof}[Proof that \eqref{eq: somme non primitif} follows from \eqref{eq: somme non primitif truncated}] It is enough to take $R=Ct$ with $C$ large enough in \eqref{eq: somme non primitif truncated}, and estimate the remaining contribution as in \S\ref{subsubsec: tail heat trace} by
		\[\sum_{\gamma\ \mathrm{n.p.}, \ d(x_0,\gamma x_0)\ge Ct} \int_{\mathcal F}H_{\widetilde X}(t,x,\gamma x)\mathrm{d}x\lesssim \mathrm{e}^{-C' t}, \]
		using pointwise Gaussian upper bounds together with the counting estimate \eqref{eq: margulis}. Here $C'$ can be taken $\ge \lambda_0$ provided that $C$ has been taken large enough above.\end{proof}
	Let us make a preliminary remark. Using the Gaussian upper bound $H(t,x,y)\lesssim  \mathrm{e}^{-\lambda_0 t-d(x,y)^2/4t}$ together with the counting estimate $N_{\rm n.p.}(T)\lesssim T^3 \mathrm{e}^{\frac \delta 2 T}$ from \eqref{eq: counting nonprimitive} below, one would get a factor $-\lambda_0+\frac{\delta^2}{4}$ in the exponential in \eqref{eq: somme non primitif}. This gives a rather trivial bound: by a result of Brooks \cite{Brooks2}, for any surface of negative curvature, one has $\lambda_0\le \frac{\delta^2}{4}$, with equality if and only if $\widetilde X$ has constant curvature --- it follows, for example, from the fact that the function $x\mapsto \mathrm{e}^{-s d(x_0,x)}$ belongs to $L^2(\widetilde X)$ whenever $s>\frac{\delta }{2}$.
	
	\bigskip Therefore, in variable negative curvature, achieving the optimal exponential decay of \eqref{eq: somme non primitif} is impossible using only Gaussian upper bounds on the heat kernel. These bounds depend only on the distance, but the level sets of the heat kernel $H(t,x,\cdot)$ may be far from being spheres centered at $x$. Thus, one has to find another strategy to bypass these radial estimates. A seemingly naive yet fruitful idea is to control the heat kernel using the Green kernel, which we introduce here.
	
	\begin{defi} For $\lambda\le \lambda_0$, introduce the \emph{Green kernel}, defined for $x\neq y$ by
		\[G_\lambda(x,y):=\int_0^{+\infty}\mathrm{e}^{\lambda t}H(t,x,y)\mathrm{d}t.\]
		For $\lambda<\lambda_0$, the function $G_\lambda(x,\cdot)$ is $\lambda$-harmonic on $\widetilde X\backslash \{x\}$, meaning that $(\Delta-\lambda)G_\lambda(x,\cdot)=0$. 
		By \cite[Lemma 2.1]{ledrappier2021local}, the Green kernel $G_{\lambda_0}(x,y)$ is finite for all $x\neq y$.
	\end{defi} 
	By the parabolic Harnack inequality of Li and Yau (see \cite[Theorem 5.3.5]{Davies_1989}), whenever $d(x,y)\ge \ell_0$ and $t\ge 1$, one has
	\[H(t,x,y)\lesssim \int_{t-1}^t H(s,x,y)\mathrm{d}s,\]
	where the implied constant is independent of $x,y,t$. Then,
	\[H(t,x,y)\lesssim \mathrm{e}^{-\lambda_0 t}\int_{t-1}^t \mathrm{e}^{\lambda_0 s}H(s,x,y)\mathrm{d}s\lesssim \mathrm{e}^{-\lambda_0 t}G_{\lambda_0}(x,y).\]
	Summing over nonprimitive elements of the groups, one finds
	\[\sum_{\gamma~\mathrm{n.p.}, \ d(x_0,\gamma x_0)\le R} H(t,x,\gamma x)\lesssim \mathrm{e}^{-\lambda_0 t}\sum_{\gamma~\mathrm{n.p.}, \ d(x_0,\gamma x_0)\le R} G_{\lambda_0}(x,\gamma x). \]
	Thus, Proposition \ref{propo: contribution non primitive truncated} will follow from this next result.
	\begin{propo}\label{propo: sum green kernel nonprimitive} For all $R\ge 1$, one has
		\begin{equation}\label{eq: sum green kernel nonprimitive}\sum_{\gamma\ \mathrm{n.p.}, \ d(x_0,\gamma x_0)\le R} \int_{\mathcal F} G_{\lambda_0}(x,\gamma x)\mathrm{d}x\lesssim R^3.\end{equation}\end{propo}
	It may appear that we have made no gain here, but the benefit of working with the Green kernel is that it shares many formal properties with the radial function $\exp(-\frac{\delta}{2} d(x,y))$, which will be sufficient to prove Proposition \ref{propo: sum green kernel nonprimitive}. These are presented in the following subsection.
	
	\subsubsection{Estimates on the Green kernel} We quote from \cite{ledrappier2021local} the tools that will be used in the proof of Proposition \ref{propo: sum green kernel nonprimitive}. The first is a Harnack inequality for $\lambda$-harmonic functions.
	\begin{propo}[Harnack inequality, see {\cite[Proposition 8.3]{ledrappier2021local}}]\label{prop: Harnack} There is a constant $C_0>0$ such that for all $\lambda\in [0,\lambda_0]$, if $f$ is a positive $\lambda$-harmonic function on an open domain $\mathcal D\subset \widetilde X$ then, for all $x\in \mathcal D$ satisfying $d(x,\partial \mathcal D)\ge 1$, one has
		\[\|\nabla \log f(x)\|\le \log C_0\]\end{propo}
	A consequence of the Harnack inequality, used extensively in \cite{ledrappier2021local}, is that for any $x\in \widetilde X$, and $y,z\in \widetilde X\backslash B(x,1)$, one has
	\[G_\lambda(x,y)\le \exp( C_0'd(y,z)) G_\lambda(x,z),\]
	for some other constant $C_0'$. The advantage of working with the Green kernel instead of the heat kernel is that $G_\lambda(x,y)$ enjoys some kind of multiplicativity property along geodesics, by the so-called \emph{Ancona--Gou{\"e}zel inequality} that we recall here.
	\begin{thm}[Ancona--Gou\"ezel inequality, see {\cite[Proposition 2.12 and Theorem 3.2]{ledrappier2021local}}] \label{thm : AnconaGouezel} There are constants $C_1,C_2,R_0>0$ such that the following holds. Let $\lambda\le \lambda_0$. Then, 
		
		\begin{enumerate}
			\item (easy side) For all $x,y,z\in \widetilde X$ with $d(x,y),d(y,z),d(x,z)\ge 1$, one has
			\[G_\lambda(x,z)G_\lambda(z,y)\le C_1 G_\lambda(x,y)\]
			\item (hard side) For all $x,y,z\in \widetilde X$ such that $y$ lies on the geodesic segment $[xz]$ and $d(x,z),d(y,z)\ge R_0$, one has the reverse inequality
			\[G_\lambda(x,y)\le C_2 G_\lambda(x,z)G_\lambda(z,y).\]
		\end{enumerate}
	\end{thm} 
	\begin{rem} ~
		
		\begin{itemize} \item 
			For $\lambda<\lambda_0$, Theorem \ref{thm : AnconaGouezel} is due to Ancona \cite{Ancona1}, whose proof was crucially based on the coercivity of the operator $\Delta-\lambda$. Uniform estimates up to $\lambda=\lambda_0$ were obtained by Ledrappier--Lim \cite{ledrappier2021local} using tools from ergodic theory, following the earlier work of Gou\"ezel for random walks on hyperbolic groups \cite{Gouezel14}. 
			\item At least for $\lambda<\lambda_0$, the hard side of Ancona--Gou\"ezel has the following probabilistic interpretation given in \cite[\S 5.3]{Ancona2}: consider a Brownian motion $\xi_t$ on $\widetilde X$. Let $x,y,z\in \widetilde X$ with $y\in [xz]$ and all distances $d(x,y),d(y,z),d(x,z)\ge 1$. Then, letting $T$ denote the first hitting time of $\partial B(z,1)$, there is a constant $c>0$ such that
			\[\mathbb P_x\big(\exists t\in [0,T], \ \xi_t\in B(y,1)\ | \ T<\infty\big)\ge c.\]
			Here $c$ is independent of $x,y,z$. This reflects the fact that Brownian trajectories tend to follow closely geodesics, which is a prominent feature in hyperbolic geometry.\end{itemize}
	\end{rem}
	An easy consequence of Ancona--Gou\"ezel inequality is the following lemma.
	\begin{lem}\label{lem: green function on the axis is independent of the point} Let $\gamma\in \Gamma\backslash\{\mathrm{id}\}$. Let $a,b$ denote two points on the axis of $\gamma$. Then,
		\[G_{\lambda_0}(a,\gamma a)\asymp G_{\lambda_0}(b,\gamma b).\]
		Here $f\asymp g$ means that there is a constant $C>0$ depending only on $X$ such that $C^{-1}f\le g\le Cf$.\end{lem}
	\begin{proof}Without loss of generality, assume that $a$ is left to $b$ on the axis of $\gamma$. Then, the points $a,b,\gamma a,\gamma b$ are placed in this order on the axis of $\gamma$. Playing with Ancona--Gou\"ezel inequality and Harnack inequality, we have, on the one hand
		\[G_{\lambda_0}(a,\gamma b)\asymp G_{\lambda_0}(a,\gamma a)G_{\lambda_0}(\gamma a,\gamma b),\]
		and on the other hand
		\[G_{\lambda_0}(a,\gamma b)\asymp G_{\lambda_0}(a,b)G_{\lambda_0}(b,\gamma b).\]
		Since $G_{\lambda_0}(\gamma a,\gamma b)=G_{\lambda_0}(a,b)$, we deduce
		\[G_{\lambda_0}(a,\gamma a)\asymp G_{\lambda_0}(b,\gamma b).\]\end{proof}
	\begin{rem}Note that in the proof above we have used both sides of Ancona--Gou\"ezel inequality.\end{rem}
	The following Proposition gives estimates on the $L^2$-mass of the Green kernel on geodesic spheres. It is originally due to Hamenst\"adt \cite{hamenstadt96}.
	\begin{propo}[{\cite[Proposition 2.16]{ledrappier2021local}}] \label{prop: L2 bound green kernel} For all $R\ge 1$ and $\lambda\le \lambda_0$, one has
		\[\int_{S(x,R)} G_\lambda^2(x,y)\mathrm{d}y\lesssim 1.\]
		Here $S(x,R)$ is the geodesic sphere centered at $x$, and $\mathrm{d}y$ denotes the restriction of the Riemannian volume form to $S(x,R)$. Thus, for all $R\ge 0$, one has
		\begin{equation}\label{eq: mass min(1,Green)} \int_{B(x,R)} \min(1,G_\lambda^2(x,y))\mathrm{d}y\lesssim R.\end{equation}
		Now, $\mathrm{d}y$ denotes the Riemannian volume form on $\widetilde X$.\end{propo}
	\begin{rem}For $\lambda<\lambda_0$, one has the stronger result
		\[\int_{\widetilde X\backslash B(x,1)} G_\lambda^2(x,y)\mathrm{d}y<+\infty.\]
		It actually follows from heat kernel estimates only. Indeed, by Fubini,
		\[\int_{\widetilde X\backslash B(x,1)} G_\lambda^2(x,y) \mathrm{d}y=\int_{s,t\ge 0} \mathrm{e}^{\lambda(t+s)}\int_{\widetilde X\backslash B(x,1)} H(t,x,y)H(s,y,x)\mathrm{d}y\mathrm{d}s\mathrm{d}t.\]
		Using the semigroup property for the heat kernel and the fact that the heat kernel is uniformly bounded for $d(x,y)\ge 1$, one finds
		\[\int_{\widetilde X\backslash B(x,1)} H(t,x,y)H(s,y,x)\mathrm{d}y\lesssim \min(H(t+s,x,x),1)\lesssim \min(\mathrm{e}^{-\lambda_0 (t+s)},1),\]
		where we used the pointwise estimate $H(t,x,x)\lesssim \mathrm{e}^{-\lambda_0 t}$ for $t\ge 1$ in the second inequality. Integrating over $s,t\ge 0$, one gets
		\[\int_{\widetilde X\backslash B(x,1)} G_\lambda^2(x,y)\mathrm{d}y\lesssim \frac{1}{(\lambda-\lambda_0)^2}.\]
		Note that the estimate blows up when $\lambda\to \lambda_0$, and it requires much more work to get uniform estimates as $\lambda$ approaches $ \lambda_0$.
	\end{rem}
	The Harnack inequality combined with the Ancona--Gou\"ezel inequality will allow us to show the following estimate.
	\begin{propo}\label{prop: comparison Green kernels}Let $x\in \widetilde X$ and $\gamma\in \Gamma\backslash\{\mathrm{id}\}$. Let $p(x)$ denote the orthogonal projection of $x$ on the axis of $\gamma$. Then,
		\begin{equation}\label{eq: comparaison green kernel trajectory} G_{\lambda_0}(x,\gamma x)\lesssim G_{\lambda_0}(p(x),\gamma p(x))\cdot \min(1,G_{\lambda_0}^2(p(x),x)).\end{equation}\end{propo}
	The idea behind Proposition \ref{prop: comparison Green kernels} is that the points $p(x)$ and $\gamma p(x)$ lie at a bounded distance from the geodesic segment $[x\gamma x]$, which allows playing with the Harnack and Ancona--Gou\"ezel inequalities to prove \eqref{eq: comparaison green kernel trajectory}.
	
	\begin{proof}[Proof of Proposition \ref{prop: comparison Green kernels}] Let $x\in \widetilde X$, and $\gamma\in \Gamma$ be different from the identity. Let $a$ denote the orthogonal projection of $x$ on the axis of $\gamma$. Let $R_1>0$, to be fixed later.
		
		\bigskip
		(1) We first assume $d(x,a)\le R_1$. In this case, one also has $d(\gamma x,\gamma a)\le R_1$. Since $d(a,\gamma a)\ge \ell_0$, we can apply Harnack inequality twice to show
		\[G_{\lambda_0}(x,\gamma x)\lesssim G_{\lambda_0}(x,\gamma a)\lesssim G_{\lambda_0}(a,\gamma a).\]
		Here the implied constants depend on $R_1$.
		
		\bigskip
		(2) We now assume $d(x,a)\ge R_1$. We let $\tilde a\in [x\gamma x]$ be the orthogonal projection of $a$ on the segment $[x\gamma x]$. 
		
		We claim that there is a constant $C_3=C_3(\kappa,\ell_0)$ such that $d(a,\tilde a)\le C_3$. Indeed, let $d=d(a,\tilde a)$, $L_1=d(x,\tilde a)$ and $L_2=d(\tilde a,\gamma x)$, so that $L_1+L_2=d(x,\gamma x)$. By Lemma \ref{lem: comparison triangles} applied to the triangles $x\tilde aa$ and $a\tilde a\gamma x$, we have a constant $C_1=C_1(\kappa,\ell_0)$ such that
		\[d(x,a)\ge L_1+d-C_1, \qquad d(a,\gamma x)\ge L_2+d-C_1\]
		Now by the triangle inequality, $d(a,\gamma x)\le \ell(\gamma)+d(\gamma a,\gamma x)$ and one finds by summing the above inequalities
		\[2d(x,a)+\ell(\gamma)\ge L_1+L_2+2d-2C=d(x,\gamma x)+2d-2C.\]
		By Lemma \ref{lem:paulinparkonnen}, one has $2d(x,a)+\ell(\gamma)\le d(x,\gamma x)+C_2$ for some $C_2=C_2(\kappa,\ell_0)$ which yields $2d\le 2C_1+C_2$, so one can take $C_3=2C_1+C_2$. Similarly, letting $\widetilde{\gamma a}$ denote the projection of $\gamma a$ on the segment $[x\gamma x]$ we find $d(\gamma a,\widetilde{\gamma a})\le C_3$.
		
		 Since $d(x,a)\ge R_1$, we obtain $d(x,\tilde a)\ge R_1-C_3$ and $d(\gamma x,\widetilde{\gamma a})\ge R_1-C_3$. We take $R_1$ large enough to ensure $R_1-C_3\ge 2R_0$ where $R_0$ is as in Theorem \ref{thm : AnconaGouezel}. If the distance between $\tilde a$ and $\widetilde{\gamma a}$ is $\le R_0$, we replace $\widetilde{\gamma a}$ by a further point $\tilde b$ on the segment $[x\gamma x]$, which satisfies $d(\tilde a,\tilde b)\ge R_0$, $d(\tilde b,\gamma x)\ge R_0$ and $d(\tilde b,\widetilde{\gamma a})\le R_0$. Note that it is always possible because $d(\gamma x,\widetilde{\gamma a})\ge 2R_0$. Then, by the Ancona--Gou\"ezel inequality of Theorem \ref{thm : AnconaGouezel}, one gets
		\[G_{\lambda_0}(x,\gamma x)\lesssim G_{\lambda_0}(x,\tilde a)G_{\lambda_0}(\tilde a,\tilde b)G_{\lambda_0}(\tilde b,\gamma x).\]
		Since $\tilde a$ lies at a bounded distance from $a$, by Proposition \ref{prop: Harnack}, one has $G_{\lambda_0}(x,\tilde a)\lesssim G_{\lambda_0}(x,a)$. Similarly, $G_{\lambda_0}(\tilde b,\gamma x)\lesssim G(\gamma a,\gamma x)$ and $G_{\lambda_0}(\tilde a,\tilde b)\lesssim G_{\lambda_0}(a,\gamma a)$. Since $G_{\lambda_0}(\gamma a,\gamma x)=G_{\lambda_0}(a,x)$, we finally obtain
		\[G_{\lambda_0}(x,\gamma x)\lesssim G_{\lambda_0}^2(x,a) G(a,\gamma a).\]
		Altogether, we have shown
		\[G_{\lambda_0}(x,\gamma x)\lesssim G_{\lambda_0}(a,\gamma a)\cdot \big(\mathbf 1_{d(x,a)\ge R_1} G_{\lambda_0}^2(x,a)+\mathbf 1_{d(x,a)\le R_1}\big).\]
		Since $G_{\lambda_0}(x,y)$ is bounded on the set $\{d(x,y)\ge R_1\}$, the term in parentheses is always $\lesssim 1$. Likewise, $G_{\lambda_0}(x,y)$ is bounded by below on the set $\{d(x,y)\le R_1\}$, thus the term in parentheses is always $\lesssim G_{\lambda_0}^2(x,a)$. This shows
		\[\mathbf 1_{d(x,a)\ge R_1} G_{\lambda_0}^2(x,a)+\mathbf 1_{d(x,a)\le R_1}\lesssim \min(1,G_{\lambda_0}^2(x,a)),\]
		and concludes the proof.
	\end{proof}
	
	As a corollary, one finds
	\begin{cor}\label{cor: comparison Green kernels}Let $\gamma\in \Gamma$ be a primitive element, $a$ be any point on the axis of $\gamma$, and $k\ge 2$ be an integer. Then, for any $x\in \widetilde X$, letting $p(x)$ denote the orthogonal projection of $x$ on the axis of $\gamma$, one has
		\[G_{\lambda_0}(x,\gamma^{k}x) \lesssim G_{\lambda_0}^2(a,\gamma^{\lfloor k/2\rfloor} a)\cdot \min(1,G_{\lambda_0}^2(x,p(x))).\]
		The implied constants in $\lesssim$ are uniform.\end{cor}
	\begin{proof}[Proof of Corollary \ref{cor: comparison Green kernels}]
		By Proposition \ref{prop: comparison Green kernels}, one has
		\[G_{\lambda_0}(x,\gamma^kx)\lesssim G_{\lambda_0}( p(x),\gamma^k p(x))\cdot \min(1,G_{\lambda_0}^2(x,p(x))).\]
		By Lemma \ref{lem: green function on the axis is independent of the point}, we get $G_{\lambda_0}(x,\gamma^k p(x))\lesssim G_{\lambda_0}(a,\gamma^k a)$. 
		
		Now, if $k=2j$, since the points $a, \gamma^j a,\gamma^{2j} a$ are aligned in this order on the axis of $\gamma$, we can use Ancona--Gou\"ezel inequality provided $d(a,\gamma^j a)\ge R_0$ to get
		\begin{equation}\label{eq: application AC} G_{\lambda_0}(a,\gamma^{2j}a)\lesssim G_{\lambda_0}(a,\gamma^{j}a)G_{\lambda_0}(\gamma^j a,\gamma^{2j}a)=G_{\lambda_0}(a,\gamma^j a)^2.\end{equation}
		In the case $d(a,\gamma^j a)\le R_0$, we use the fact that the Green kernel is bounded by positive constants by above and below on the set $\{\ell_0\le d(x,y)\le 2R_0\}$ to still recover \eqref{eq: application AC}.
		
		Similarly, if $k=2j+1$, one has
		\[ \qquad G_{\lambda_0}(a,\gamma^{2j+1}a)\lesssim G_{\lambda_0}(a,\gamma^j a)G_{\lambda_0}(a,\gamma^{j+1}a).\]
		If $d(a,\gamma^j a)\ge R_0$, it follows from Ancona--Gou\"ezel inequality, while if $d(a,\gamma^j a)\le R_0$ it follows from the same argument as above. 
		Lastly, one has $G_{\lambda_0}(a,\gamma^{j+1}a)\lesssim G_{\lambda_0}(a,\gamma^j a)$. If $d(a,\gamma a)\ge R_0$ it follows from Ancona--Gou\"ezel inequality and the boundedness of the Green kernel on $\{d(x,y)\ge 1\}$, since then
		\[G_{\lambda_0}(a,\gamma^{j+1}a)\lesssim G_{\lambda_0}(a,\gamma^ja)G_{\lambda_0}(a,\gamma a)\lesssim G_{\lambda_0}(a,\gamma^j a),\]
		while if $d(a,\gamma a)\le R_0$, it follows directly from Harnack inequality.
	\end{proof}
	
	\begin{rem}As we discussed earlier, all the properties of $G_{\lambda_0}$ listed above also hold for the function $f(x,y):=\mathrm{e}^{-\frac{\delta}{2} d(x,y)}$, where $\delta$ is the critical exponent of $\Gamma$. The easy side of Ancona inequality corresponds to the triangle inequality, while the hard side corresponds to the case of equality in the triangle inequality. The $L^2$ estimate $\int_{S(x,R)}f^2 \lesssim 1$ in Proposition \ref{prop: L2 bound green kernel} mirrors the estimate on the area of spheres $\mathrm{Area}(S(x,R))\lesssim \mathrm{e}^{\delta R}$. Likewise, Corollary \ref{cor: comparison Green kernels} corresponds to the quadrilateral comparison result of Lemma \ref{lem:paulinparkonnen}. \end{rem}

	\subsubsection{Proof of Proposition \ref{propo: sum green kernel nonprimitive}} We now combine the tools of the previous subsection to prove \eqref{eq: sum green kernel nonprimitive}. 
	
	We consider a subset $\mathcal P\subset \Gamma$ of primitive elements, such that any nonprimitive element of $\Gamma$ writes uniquely $g^{-1}\gamma^k g$ for some $\gamma\in \mathcal P$, $k\ge 2$ and $g\in \langle \gamma\rangle \backslash \widetilde X$. We can choose $\mathcal P$ such that the axes of all elements of $\mathcal P$ pass through the fundamental domain $\mathcal F$ (it is possible because $\mathrm{axis}(g^{-1}\gamma g)=g^{-1}\mathrm{axis}(\gamma)$). This allows us to write, for any $x\in \mathcal F$,
	\[\sum_{\gamma\ \mathrm{n.p.}} G_{\lambda_0}(x,\gamma x)\mathbf 1_{d(x_0,\gamma x_0)\le R}= \sum_{\gamma\in \mathcal P} \sum_{k\ge 2}\sum_{g\in \langle \gamma\rangle\backslash \Gamma} G_{\lambda_0}(x,g^{-1}\gamma^k gx)\mathbf 1_{d(x_0,g^{-1}\gamma^k g x_0)\le R}.\]
	Since $x_0\in \mathcal F$ and the axis of any $\gamma\in \mathcal P$ passes through $\mathcal F$, one has 
	\[d(x_0,g^{-1}\gamma^k gx_0)=d(gx_0,\gamma^k gx_0)\ge \ell(\gamma^k) \ge d(x_0,\gamma^k x_0)-\mathrm{diam}(\mathcal F).\]
	Moreover, for any $x\in \mathcal F$, one has $d(x,g^{-1}\gamma^k gx)\ge 
	C(\mathrm{diam}(\mathcal F),\ell_0) d(x_0,g^{-1}\gamma^k gx_0)$. Therefore, there is a constant $C>0$ such that
	\[\sum_{\substack{\gamma\ \mathrm{n.p.} \\ d(x_0,\gamma x_0)\le R}} G_{\lambda_0}(x,\gamma x)\le \sum_{\substack{\gamma\in \mathcal P, \ k\ge 2 \\ d(x_0,\gamma^kx_0)\le CR}}\sum_{g\in \langle \gamma\rangle\backslash \Gamma} G_{\lambda_0}(x,g^{-1}\gamma^k gx)\mathbf 1_{d(x,g^{-1}\gamma^k g x)\le CR}\]
	We integrate over $\mathcal F$, and perform a change of variable $x\mapsto gx$ in every term of the sum on the right-hand side. When $g$ runs over $\langle \gamma\rangle\backslash \Gamma$, the sets $\{g\mathcal F\}$ form a partition --- up to a set of measure zero --- of a fundamental domain for the action of $\langle \gamma\rangle$ on $\widetilde X$, denoted $\mathcal F_\gamma$. Eventually,
	\[\sum_{\substack{\gamma \ \mathrm{n.p.} \\ d(x_0,\gamma x_0)\le R}} \int_{\mathcal F} G_{\lambda_0}(x,\gamma x) \mathrm{d}x\le \sum_{\substack{\gamma \in \mathcal P, \  k\ge 2 \\ d(x_0,\gamma^k x_0)\le CR}}\int_{\mathcal F_\gamma} G_{\lambda_0}(x,\gamma^k x) \mathbf 1_{d(x,\gamma^k x)\le CR}\mathrm{d}x.\]
	Since the integral on the right-hand side does not depend on the fundamental domain $\mathcal F_\gamma$, we can choose the following one: we fix an arbitrary point $a\in \mathrm{axis}(\gamma)\cap \mathcal F$, then let $\mathcal F_\gamma$ be the ensemble of points located between the two geodesic lines that cross the axis of $\gamma$ orthogonally at the points $a$ and $\gamma a$, as described in Figure \ref{fig: fundamental domain gamma}.
	
	\begin{figure} \label{fig: fundamental domain gamma}
		\frame{\includegraphics[width=10cm]{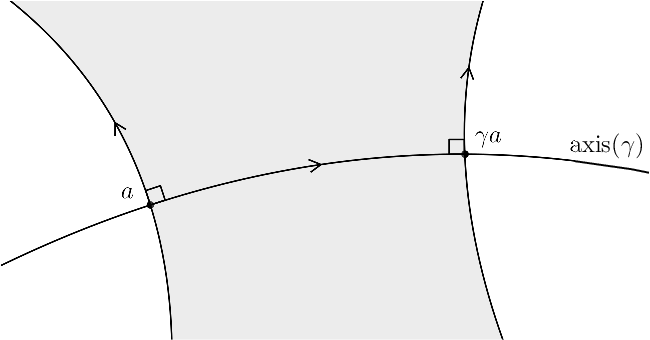}}
		\caption{The shaded area represents the fundamental domain $\mathcal F_\gamma$ for the action of $\langle \gamma\rangle$ on $\widetilde X$.}
	\end{figure} Denote by $p(x)$ the projection of $x$ on the axis of $\gamma$. Then, by Corollary \ref{cor: comparison Green kernels}, one has
	\[G_{\lambda_0}(x,\gamma^k x)\lesssim G_{\lambda_0}^2(a,\gamma^{\lfloor k/2\rfloor }a) \cdot \min(1,G_{\lambda_0}^2(x,p(x))).\]
	Moreover, since $a$ and $x_0$ both belong to $\mathcal F$, Harnack inequality implies
	\[G_{\lambda_0}(a,\gamma^{\lfloor k/2\rfloor} a)\lesssim G_{\lambda_0}(x_0,\gamma^{\lfloor k/2\rfloor} x_0).\]
	Also, by Lemma \ref{lem:paulinparkonnen}, there is a constant $c>0$ such that $d(x,\gamma^k x)\ge c d(x,\mathrm{axis}(\gamma))$. Up to increasing the constant $C$, we find
	\begin{multline}\label{eq: int G lambda x gamma x}
		\int_{\mathcal F_\gamma} G_{\lambda_0}(x,\gamma^k x)\mathbf 1_{d(x,\gamma^k g)\le CR}\mathrm{d}x\\ \lesssim G_{\lambda_0}^2(x_0,\gamma^{\lfloor k/2\rfloor} x_0)\int_{\mathcal F_\gamma} \min(1,G_{\lambda_0}^2(x,p(x))) \mathbf 1_{d(x,p(x))\le CR} \mathrm{d}x.
	\end{multline}
	To estimate the integral in \eqref{eq: int G lambda x gamma x}, place uniformly spaced points $x_1,\ldots,x_N$ on the segment $a\gamma a$, with $N\le \ell(\gamma)\le R$. Let $x\in \mathcal F_\gamma$ be at distance $\le CR$ from the axis of $\gamma$. Then, one can find $i\in\{1,\ldots,N\}$ such that $d(p(x),x_i)\lesssim 1$ (implying $x\in B(x_i,CR+1))$. It follows from Harnack inequality that $G_{\lambda_0}(x,p(x))\lesssim G_{\lambda_0}(x_i,x)$. We infer
	\[\int_{\mathcal F_\gamma} \min(1,G_{\lambda_0}^2(x,p(x)))\mathbf 1_{d(x,p(x))\le CR} \mathrm{d}x\lesssim \sum_{i=1}^N \int_{B(x_i,CR+1)} \min(1,G_{\lambda_0}^2(x_i,x))\mathrm{d}x.\]
	It follows from \eqref{eq: mass min(1,Green)} that
	\[\int_{\mathcal F_\gamma}\min(1,G_{\lambda_0}^2(x,p(x)))\mathbf 1_{ d(x,p(x))\le CR} \mathrm{d}x\lesssim N R\lesssim R^2.	\]
	Recalling \eqref{eq: int G lambda x gamma x}, one finds 
	\[\int_{\mathcal F_\gamma} G_{\lambda_0}(x,\gamma^k x)\mathbf 1_{d(x,\gamma^k x)\le CR}\mathrm{d}x\lesssim R^2G_{\lambda_0}^2(x_0,\gamma^{\lfloor k/2\rfloor} x_0).\]
	Summing over elements of $\mathcal P$ and $k\ge 2$, one obtains
	\[\sum_{\substack{\gamma~\rm n.p. \\ d(x_0,\gamma x_0)\le R}} \int_{\mathcal F_\gamma} G_{\lambda_0}(x,\gamma x) \mathrm{d}x\lesssim R^2\sum_{\substack{\gamma\in \mathcal P,\ k\ge 2\\  d(x_0,\gamma^k x_0)\le CR}} G_{\lambda_0}^2(x_0,\gamma^{\lfloor k/2\rfloor}x_0).\]
	Each nontrivial element of $\Gamma$ appears at most twice on the right-hand side, which is thus bounded by a constant times
	\[R^2\sum_{\substack{\gamma\in \Gamma\backslash \{\mathrm{id}\}\\ d(x_0,\gamma x_0)\le CR}} G_{\lambda_0}^2(x_0,\gamma x_0)\lesssim R^2 \int_{B(x_0,CR)\backslash B(x_0,1)} G_{\lambda_0}^2(x_0,y)\mathrm{d}y\lesssim R^3,\]
	Here we have used Harnack inequality and Proposition \ref{prop: L2 bound green kernel} again. This terminates the proof.
	
	\begin{rem}By replacing $G_{\lambda_0}(x,y)$ by $\exp(-\frac{\delta}{2}d(x,y))$ in the proofs above, one obtains as a byproduct the following counting estimate, which is, to our knowledge, new in variable negative curvature (and might be helpful elsewhere).
		
		\begin{propo}\label{prop: comptage non primitif}Fix $x\in \widetilde X$ and let
			\[N_{\rm n.p.}(x,R)\stackrel{\rm{def}}{=}\# \big\{ \gamma\in \Gamma_{\rm n.p.} \ : \ d(x,\gamma x)\le R\big\}.\]
			Then there is a positive constant $C_x>0$ such that for all $R$ large enough, 
			\begin{equation}\label{eq: counting nonprimitive} N_{\rm n.p.}(x,R)\le C_x R^3\mathrm{e}^{\frac{\delta R}{2}}.\end{equation}
	\end{propo}\end{rem}
	\section{Strong convergence and heat kernels} \label{sec: strong convergence} In the following, the operator norm of a bounded operator $A:H\to H$ is denoted by $\|A\|_H$.
	
	\subsection{About strong convergence}
	We recall that a sequence of finite-dimensional unitary representations $\left(\rho_i,V_i\right)$ of a discrete group $\Gamma$ is said to strongly converge to the (right) regular representation $\left(\lambda_\Gamma,\ell^2(\Gamma)\right)$ if for any $z\in \mathbb{C}\left[\Gamma\right]$,
	\begin{align}\label{eq:strong-conv}
		\lim_{i\to \infty}\|\rho_i(z)\|_{V_i}= \|\lambda_\Gamma(z)\|_{\ell^2(\Gamma)}.
	\end{align} We refer the reader to a recent survey of Magee \cite{Magee_survey} for historical background and an overview of recent developments.
	\\
	Throughout this section we let $X=\Gamma\backslash\widetilde{X}$ be a
	fixed closed Riemannian surface with strictly negative curvature. Let $V_{n}^{0}\stackrel{\rm{def}}{=} \ell_{0}^{2}\left(\left\{ 1,\dots,n\right\} \right)$ be
	the space of square summable functions with $0$-mean. For $\phi\in\textup{Hom}\left(\Gamma,\mathcal S_{n}\right)$
	we have defined \begin{equation}
		\rho_{\phi}\stackrel{\rm{def}}{=}\text{std}_{n-1}\circ\phi \label{eq:factor-th}
	\end{equation}
	on $V_{n}^{0}$
	where we recall that $\text{std}_{n-1}$ is the standard $n-1$ dimensional irreducible
	representation. We write $\mathcal{K}\left(L^{2}\left(X\right)\right)$
	to denote the space of compact operators on $L^{2}\left(X\right)$.
	The following Theorem of Louder and Magee guarantees the existence of strongly convergent representations of $\Gamma$ of the form \eqref{eq:factor-th}.
	\begin{thm}[Louder--Magee, {\cite[Corollary 1.2]{louder2022strongly}}]
		\label{thm:Louder-Magee}There exists a sequence $\left\{ \phi_{i}\right\} _{i\in\mathbb{N}}$
		where $\phi_{i}\in\textup{Hom}\left(\Gamma,\mathcal S_{n_{i}}\right)$ such
		that, letting $\rho_i\stackrel{\rm{def}}{=}\rho_{\phi_i}$, the sequence of representations $\left(\rho_i,V_{n_{i}}^{0}\right)$ converges strongly
		to $\left(\lambda_\Gamma,\ell^{2}(\Gamma)\right)$ as $i\to\infty$. In
		particular, for any finitely supported map $\gamma\mapsto a_{\gamma}\in\mathcal{K}\left(L^{2}\left(X\right)\right)$,
		we have 
		\begin{equation}
			\Big\|\sum_{\gamma\in\Gamma}a_{\gamma}\otimes\rho_{i}(\gamma)\Big\|_{L^{2}\left(X\right)\otimes V_{n_i}^{0}}\to \Big\|\sum_{\gamma\in\Gamma}a_{\gamma}\otimes\lambda_\Gamma(\gamma)\Big\|_{L^{2}\left(X\right)\otimes \ell^2\left(\Gamma\right)}.\label{eq:mat-strong-conv}
		\end{equation}
	\end{thm}
	The conclusion \eqref{eq:mat-strong-conv} follows from \eqref{eq:strong-conv} by matrix amplification (e.g. \cite[Proposition 3.3]{Magee_survey}) and approximating on both sides by finite rank operators.
	Given the sequence $\left\{ \phi_{i}\right\} _{i\in\mathbb{N}}$ guaranteed
	by Theorem \ref{thm:Louder-Magee}, we write $X_{i}$ to be the degree
	$n_{i}$ covering space corresponding to $\phi_{i}$. 
	
	\subsection{Function spaces}
	We define $L_{\text{new}}^{2}\left(X_{i}\right)$ to be the space
	of $L^{2}$ functions on $X_{i}$ orthogonal to all lifts of $L^{2}$
	functions from $X$. Then $L^2(X_i)$ splits as the direct sum of two closed orthogonal subspaces
	\[
	L^{2}\left(X_{i}\right)\cong L_{\text{\text{new}}}^{2}\left(X_i\right)\oplus L^{2}\left(X\right).
	\]
	Moreover, the Laplacian acts diagonally on this direct sum. More precisely, The space $L^2_{\rm new}(X_i)\cap \mathrm{Dom}(\Delta_{X_i})$ is preserved by the action of $\Delta_{X_i}$, indeed for any $f$ in this space and $g\in C^\infty(X)$, letting $\pi:X_i\to X$ denote the projection one has, since $\Delta_{X_i}$ is symmetric and $\Delta_{X_i}\circ \pi^*=\pi^* \circ \Delta_X$,
	\[\langle \Delta_{X_i} f, g\circ \pi\rangle_{L^2(X_i)}=\langle f,\Delta_{X_i}(g\circ \pi)\rangle_{L^2(X_i)}=\langle f, (\Delta_X g)\circ \pi\rangle_{L^2(X_i)}=0.\]
	By density, this holds for all $g\in L^2(X)$.
	
	The set of new eigenvalues is exactly given by the spectrum of ${\Delta_{X_i}}_{|L^2_{\rm new}(X_i)}$. In particular, one has
	\[
	\lambda_{1}(X_{i})=\inf\text{\ensuremath{\left(\mathrm{Sp}\left({\Delta_{X_i}}_{|L_{\text{new}}^{2}\left(X_{i}\right)}\right)\right)}},
	\]
	our goal is to prove that
	\begin{equation}
		\lambda_{1}\left(X_{i}\right)\geqslant\lambda_{0}-o_{i\to\infty}(1).\label{eq:goal}
	\end{equation}
	As before, we let $\mathcal{F}$ be a bounded fundamental domain for $X$. Let $C^{\infty}\big(\widetilde{X};V_{n_i}^{0}\big)$
	denote the space of smooth $V_{n_i}^{0}$-valued functions on $\widetilde{X}$. There
	is an isometric linear isomorphism between 
	\[
	C^{\infty}\left(X_{i}\right)\cap L_{\text{\text{new}}}^{2}\left(X_{i}\right),
	\]
	and the space of smooth $V_{n_i}^{0}$-valued functions on $\widetilde{X}$
	satisfying 
	\[
	f(\gamma x)=\rho_{i}(\gamma)f(x),
	\]
	for all $\gamma\in\Gamma$, with finite norm
	\[
	\|f\|_{2}^{2}\stackrel{\rm{def}}{=}\int_{\mathcal{F}}|f(x)|_{V_{n_i}^{0}}^{2}\mathrm d x<\infty.
	\]
	We denote the space of such functions by $C_{\phi_{i}}^{\infty}\big(\widetilde{X};V_{n_i}^{0}\big).$
	The completion of $C_{\phi_{i}}^{\infty}\big(\widetilde{X};V_{n_i}^{0}\big)$
	with respect to $\|\bullet\|_{2}$ is denoted by $L_{\phi_{i}}^{2}\big(\widetilde{X};V_{n_i}^{0}\big)$;
	the isomorphism above extends to one between $L_{\text{\text{new}}}^{2}\big(X_{i}\big)$
	and $L_{\phi_{i}}^{2}\big(\widetilde{X};V_{n_i}^{0}\big)$.

	\subsection{Proof of Theorem \ref{main2}}
	We will now prove Theorem \ref{main2}, using the above tools.
	\begin{proof}
		Consider the heat operator $\exp\left(-t\Delta\right):L_{\text{new}}^{2}\left(X_{i}\right)\to L_{\text{new}}^{2}\left(X_{i}\right)$.

		It is enough for our purposes to just consider time $t=1$. By the spectral theorem, we have
		\[
		\exp\left(-\lambda_{1}\left(X_{i}\right)\right)=\|\exp\left(-\Delta\right)\|_{L_{\text{new}}^{2}\left(X_{i}\right)}.
		\]
		As in \cite[Section 6.1]{HM}  there is an
		isomorphism of Hilbert spaces
		\begin{align*}
			L_{\phi_{i}}^{2}\big(\widetilde{X};V_{n_i}^{0}\big) & \cong L^{2}(\mathcal F)\otimes V_{n_i}^{0};\\
			f\mapsto & \sum_{e_{j}}\langle f\lvert_{\mathcal F},e_{j}\rangle_{V_{n_i}^{0}}\otimes e_{j},
		\end{align*}
		where $\{e_j\}$ is a finite orthonormal basis of $V_{n_i}^0$. Under this isomorphism, one has
		\[
		\exp\left(-\Delta\right)\cong\sum_{\gamma\in\Gamma}a_{\gamma}\otimes\rho_{i}\left(\gamma\right),
		\]
		where $a_{\gamma}  :L^{2}\left(\mathcal F\right)\to L^{2}\left(\mathcal F\right)$ is defined by
		\begin{align*}
			a_{\gamma}[f](x) & \stackrel{\rm{def}}{=}\int_{y\in\mathcal{F}}H_{\widetilde{X}}(x,\gamma y)f(y)~\mathrm dy.
		\end{align*}
		Here $H_{\widetilde{X}}(x,y)\stackrel{\rm{def}}{=} H_{\widetilde{X}}(1,x,y)$ denotes the heat
		kernel at time $t=1$. Introduce a cutoff function
		\[
		\chi_{T}(r)=\begin{cases}
			1 & \text{for \ensuremath{r\in\left[0,T\right]} }\\
			0 & \text{for }r>T
		\end{cases}.
		\]
		Then, split
		\[a_\gamma=a_\gamma^{\chi_T}+a_\gamma^{1-\chi_T},\]
		with
		\[a_\gamma^{\chi_T}[f](x)\stackrel{\rm{def}}{=}\int_{\mathcal{F}}H_{\widetilde{X}}(x,\gamma y)\chi_{T}\left(d(x,\gamma y)\right)f(y)\mathrm dy\]
		and similarly
		\[a_\gamma^{1-\chi_T}[f](x)\stackrel{\rm{def}}{=}\int_{\mathcal{F}}H_{\widetilde{X}}(x,\gamma y)\big(1-\chi_{T})\left(d(x,\gamma y)\right)f(y)\mathrm dy.\]
		It follows that 
		\[
		\|\exp\left(-\Delta_{X_{i}}\right)\|_{L_{\text{new}}^{2}\left(X_{i}\right)}\leqslant\underbrace{\Big\|\sum_{\gamma\in\Gamma}a_{\gamma}^{\chi_{T}}\otimes\rho_i(\gamma)\Big\|_{L^2(\mathcal F)\otimes V_{n_i}^{0}}}_{\rm (a)}+\underbrace{\Big\|\sum_{\gamma\in\Gamma}a_{\gamma}^{1-\chi_{T}}\otimes\rho_i(\gamma)\Big\|_{L^2(\mathcal F)\otimes V_{n_i}^{0}}}_{\rm (b)}.
		\]
		We first bound $\rm{(b)}$. Because each $\rho_{i}$ is unitary, one has
		\begin{equation}
			\Big\|\sum_{\gamma\in\Gamma}a_{\gamma}^{1-\chi_{T}}\otimes\rho_i(\gamma)\Big\|_{L^{2}(\mathcal F)\otimes V_{n_i}^{0}}\leqslant\sum_{\gamma\in\Gamma}\big\|a_{\gamma}^{1-\chi_{T}}\big\|_{L^{2}(\mathcal F)}.\label{eq:op-bound-base}
		\end{equation}
		We recall Schur's inequality : if $A$ is a linear operator on $L^2(\mathcal F)$ with Schwartz kernel $K(x,y)$, then, letting
		\[C_1=\sup_{x\in \mathcal F}\int_{\mathcal F} |K(x,y)|~\mathrm dy, \qquad C_2=\sup_{y\in \mathcal F}\int_{\mathcal F} |K(x,y)|~\mathrm dx,\]
		one has
		\[\|A\|_{L^2(\mathcal F)}\le \sqrt{C_1C_2}.\]
		For the operator $a_\gamma^{1-\chi_T}$, this gives
		\begin{align*}
			\big\|a_{\gamma}^{1-\chi_{T}}\big\|_{L^{2}(\mathcal F)} & \le \sup_{x\in\mathcal{F}}\int_{\mathcal{F}}H_{\widetilde{X}}(x,\gamma y)\left(1-\chi_{T}\right)(d(x,\gamma y))\mathrm dy\\
			& \leqslant\text{Vol}\left(X\right)\cdot\sup_{x,y\in\mathcal{F}}\left(H_{\widetilde{X}}(x,\gamma y)\left(1-\chi_{T}\right)\left(d(x,\gamma y)\right)\right).
		\end{align*}
		By Gaussian upper bounds for $t=1$, we have
		\[
		H_{\widetilde{X}}(x,\gamma y)\lesssim\exp\left(-\frac{d(x,\gamma y)^{2}}{8}\right).
		\]
		Fix $x_0\in \mathcal F$, then for $T$ large enough, since $\mathcal F$ is bounded, we have 
		$$d(x,\gamma y)^2 \mathbf 1_{d(x,y)\ge T}\ge \frac 12 d(x_0,\gamma x_0) \mathbf 1_{d(x_0,y_0)\ge \frac T2},$$ 
		valid for any $x,y\in \mathcal F$, so that
		\[\big\|a_{\gamma}^{1-\chi_{T}}\big\|_{L^{2}(\mathcal F)}\lesssim \mathbf 1_{d(x_0,\gamma x_0)\ge \frac T2}\mathrm{e}^{-\frac{d(x_0,\gamma x_0)^2}{16}}.\]
		Summing over all contributions, it follows that $\mathrm{(b)}$ is bounded by a constant times
		\[\sum_{\gamma \in \Gamma, d(x_0,\gamma x_0)\ge T/2}\mathrm{e}^{-\frac{d(x_0,\gamma x_0)^2}{16}}\]
		This kind of sum has already been estimated in \S \ref{subsec: contribution primitive}, and we get
		\[
		\mathrm{(b)} \lesssim \int_{T/2}^{\infty}\mathrm{e}^{\delta r-\frac{r^{2}}{16}}\mathrm dr\lesssim \mathrm{e}^{-cT^{2}}
		\]
		for some positive constant $c$. 
		
		We now look at $\mathrm{(a)}$. Since for any fixed $T>0$, $a_{\gamma}^{\chi_{T}}$
		is non-zero for only finitely many $\gamma\in\Gamma$, by the conclusion
		of Theorem \ref{thm:Louder-Magee}, we have that
		\[
		\Big\|\sum_{\gamma\in\Gamma}a_{\gamma}^{\chi_{T}}\otimes\rho_i(\gamma)\Big\|_{L^{2}(\mathcal F)\otimes V_{n_i}^{0}}\leqslant\Big\|\sum_{\gamma\in\Gamma}a_{\gamma}^{\chi_{T}}\otimes\lambda_\Gamma(\gamma)\Big\|_{L^{2}(\mathcal F)\otimes V_{n_i}^{0}}\left(1+o_{i\to\infty}(1)\right).
		\]
		As in \cite[Section 6.2]{HM} there is an
		isomorphism of Hilbert spaces 
		\begin{align*}
			L^{2}\left(\mathcal{F}\right)\otimes\ell^{2}\left(\Gamma\right) & \cong L^{2}(\widetilde{X}),\\
			f\otimes\delta_{\gamma} & \mapsto f\circ\gamma^{-1},
		\end{align*}
		with $f\circ\gamma^{-1}$ extended by zero outside of $\gamma\mathcal{F}$.
		Under this isomorphism, the operator $\sum_{\gamma\in\Gamma}a_{\gamma}^{\chi_{T}}\otimes\lambda_\Gamma(\gamma)$
		is conjugated to $\text{Op}_{H_{\widetilde{X}}\chi_{T}}:L^{2}(\widetilde{X})\to L^{2}(\widetilde{X})$
		where
		\[
		[\text{Op}_{H_{\widetilde{X}}\chi_{T}}f](x)\stackrel{\rm{def}}{=}\int_{\widetilde{X}}H_{\widetilde{X}}(x,y)\chi_{T}(d(x,y))f(y)~\mathrm dy.
		\]
		Similarly we define 
		\[
		[\text{Op}_{H_{\widetilde{X}}(1-\chi_{T})}f](x)\stackrel{\rm{def}}{=}\int_{\widetilde{X}}H_{\widetilde{X}}(x,y)\left(1-\chi_{T}\right)(d(x,y))f(y)~\mathrm dy
		\]
		and we see that 
		\begin{align*}\big\|\text{Op}_{H_{\widetilde{X}}\chi_{T}}\big\|_{L^2\left(\widetilde{X}\right)} & =\big\|\exp(-\Delta_{\widetilde{X}})-\text{Op}_{H_{\widetilde{X}}(1-\chi_{T})}\big\|_{L^2\left(\widetilde{X}\right)} \\ &
			\leqslant \big\|\exp(-\Delta_{\widetilde{X}})\big\|_{L^2\left(\widetilde{X}\right)}+\big\|\text{Op}_{H_{\widetilde{X}}(1-\chi_{T})}\big\|_{L^2(\widetilde{X})} \\ & =\mathrm{e}^{-\lambda_{0}(\widetilde{X})}+\big\|\text{Op}_{H_{\widetilde{X}}(1-\chi_{T})}\big\|_{L^2(\widetilde{X})}.
		\end{align*}
		Now it remains to bound $\big\|\text{Op}_{H_{\widetilde{X}}(1-\chi_{T})}\big\|_{L^2(\widetilde{X})}$.
		
		Using Schur's inequality again, one gets 
		\begin{equation}
			\big\|\text{Op}_{H_{\widetilde{X}}(1-\chi_{T})}\big\|_{L^2\left(\widetilde{X}\right)}\leqslant\sup_{x\in\widetilde{X}}\int_{\widetilde{X}}H_{\widetilde{X}}(x,y)\left(1-\chi_{T}\right)(d(x,y))~ \mathrm dy.\label{eq:op-norm-bound}
		\end{equation}
		For each fixed $x\in\widetilde{X}$, using the rough bound $H_{\widetilde X}(x,y)\lesssim \mathrm{e}^{-d(x,y)^2/8}$, we have
		\[\int_{\widetilde{X}}H_{\widetilde{X}}(x,y)\left(1-\chi_{T}\right)(d(x,y))\mathrm dy\le \sum_{k\ge \lfloor T \rfloor} \int_{B(x,k+1)\backslash B(x,k)} \mathrm{e}^{-d(x,y)^2/8} \mathrm dy.\]
		Now each term of the sum is bounded by $\mathrm{Vol}_{\tilde g} (B(x,k+1))\cdot\mathrm{e}^{-k^2/8}\lesssim \mathrm{e}^{\delta k-k^2/8}$ and summing over $k\ge \lfloor T\rfloor$ gives that 
		\[
		\big\|\text{Op}_{H_{\widetilde{X}}(1-\chi_{T})}\big\|_{L^2(\widetilde{X})}\lesssim \mathrm{e}^{-c'T^2},
		\]
		for some positive constant $c'$.
		
		Now for any $\varepsilon>0$ one
		can choose $T>0$ sufficiently large (but fixed) and run the above
		arguments to deduce that for sufficiently large $i$ (given $\varepsilon$)
		\[
		\mathrm{e}^{-\lambda_{1}\left(X_{i}\right)}\leqslant \mathrm{e}^{-\lambda_{0}(\widetilde{X})}\left(1+\varepsilon\right),
		\]
		which implies that 
		\[
		\lambda_{1}\left(X_{i}\right)\geqslant\lambda_{0}(\widetilde{X})-\log\left(1+\varepsilon\right),
		\]
		and thus proving (\ref{eq:goal}).
	\end{proof}

\end{document}